\documentclass[11pt]{amsart}
\usepackage{amssymb, amsfonts, amsmath, amsthm}
\usepackage{enumerate}
\usepackage[mathscr]{eucal}
\usepackage{mathrsfs}
\usepackage{pgf,tikz}
 \usetikzlibrary{arrows}
\usepackage{url}

\newtheorem{theorem}{Theorem}
\newtheorem{corollary}[theorem]{Corollary}
\newtheorem{lemma}[theorem]{Lemma}
\newtheorem{proposition}[theorem]{Proposition}

\theoremstyle{definition}
\newtheorem{definition}[theorem]{Definition}
\newtheorem{example}[theorem]{Example}

\newtheorem{algorithm}[theorem]{Algorithm}

\title{Almost symmetric numerical semigroups with high type}

\author{P.A. Garc\'{\i}a-S\'anchez}
\address{IEMath-Gr and Departamento de \'Algebra\newline
\indent Universidad de Granada\newline 
\indent E-18071 Granada, Spain}
\email{pedro@ugr.es}

\author{I. Ojeda}
\address{IMUEx - Departamento de Matem\'aticas\newline
\indent Universidad de Extremadura\newline  
\indent E-06071 Badajoz, Spain}
\email{ojedamc@unex.es}

\date{\today}

\begin{document}

\thanks{
The first author was partially supported by the Junta de Andaluc\'{\i}a research group FQM-366, and by the project MTM2017-84890-P (MINECO/FEDER, UE)}

\thanks{
The second author was partially supported by the research groups FQM-024 (Junta de Extremadura/FEDER funds) and by the project MTM2015-65764-C3-1-P (MINECO/FEDER, UE) and by the project MTM2017-84890-P (MINECO/FEDER, UE)}

\date{\today}% to be added afterwards % revised date 
% AMS Subject Classification (2010)

\begin{abstract}
We establish a one-to-one correspondence between numerical semigroups of genus $g$ and almost symmetric numerical semigroups with Frobenius number $F$ and type $F-2g$, provided that $F\ge 4g+1$.
\end{abstract}

\maketitle

\section{Introduction}

Let $\mathbb{N}$ denote the set of nonnegative integers. A \emph{numerical semigroup} is a submonoid of $(\mathbb N,+)$ with finite complement in $\mathbb N$. If $S$ is a numerical semigroup, the set $\mathbb N\setminus S$ is known as the set of gaps of $S$. Its cardinality is the \emph{genus} of $S$, denoted here $\operatorname{g}(S)$. The \emph{multiplicity} of a numerical semigroup, $\operatorname{m}(S)$, is the smallest positive integer not belonging to its gap set. The largest integer not belonging to $S$ is the \emph{Frobenius number} of $S$, denoted $\operatorname{F}(S)$.

Associated to $S$ we can define the following order relation: for $a,b\in \mathbb{Z}$, $a\le_S b$ if $b-a\in S$. The set of maximal elements of $\mathbb{Z}\setminus S$ with respect to $\le_S$ is the set of \emph{pseudo-Frobenius} numbers of $S$, $\operatorname{PF}(S)$; its cardinality is the \emph{type} of $S$, denoted $\operatorname{t}(S)$.

A numerical semigroup $S$ is irreducible if it cannot be written as the intersection of two numerical semigroups properly containing $S$. This is equivalent to say that $S$ is maximal (with respect to set inclusion) in the set of numerical semigroups having Frobenius number equal to $\operatorname{F}(S)$.  If $\operatorname{F}(S)$ is odd, this is equivalent to $\operatorname{g}(S)=(\operatorname{F}(S)+1)/2$ ($S$ is \emph{symmetric}); while if $\operatorname{F}(S)$ is even, this is the same as to impose $\operatorname{g}(S)=(\operatorname{F}(2)+2)/2$ ($S$ is pseudo-symmetric, see for instance \cite[Chapter 3]{ns-book}, or \cite{ns-app} for the history behind the names of the invariants defined above). For symmetric numerical semigroups the type is one (and this precisely characterizes them), and for pseudo-symmetric numerical semigroups the type is two (though this does not characterize this property). Thus for any irreducible numerical semigroup $S$, the equality $\operatorname{g}(S)=(\operatorname{F}(S)+\operatorname{t}(S))/2$ holds. Then converse is not true, but gives rise to a wider family of numerical semigroups: almost-symmetric numerical semigroups. A numerical semigroup is said to be \emph{almost symmetric} provided that $\operatorname{g}(S)=(\operatorname{F}(S)+\operatorname{t}(S))/2$. It is well known that for any numerical semigroup $\operatorname{g}(S)\ge (\operatorname{F}(S)+\operatorname{t}(S))/2$ (see \cite[Proposition 2.2]{nari}), and so almost symmetric numerical semigroups are those attaining the equality. Indeed, as shown in \cite{nari} almost symmetric numerical semigroups have some symmetry properties.

Almost symmetric numerical semigroups have attracted the attention of many researchers, not only because they generalize the irreducible property in numerical semigroups, but also because they rise in a natural way as a generalization of the Gorenstein property in one-dimensional rings (see \cite{barucci}). Many papers deal with the almost symmetric property and how to construct examples of these semigroups (see for instance \cite{bs, as-ft, hw} and the references there in). Some manuscripts like \cite{hw, mosc} and \cite{sw} deal with almost symmetric numerical semigroups with small type and small embedding dimension, which is the cardinality of a minimal generating set of the numerical semigroup. The semigroups with consider in this manuscript have large type.

The original aim of this note was to clarify a computational evidence noticed by the second author when using the algorithms given in \cite{as-ft} and implemented in the \texttt{GAP} \cite{gap} package \texttt{NumericalSgps} \cite{numericalsgps}  (see \cite[Remark 5.1]{as-ft} for further details). To this end, we present a one to one correspondence with numerical semigroups with given genus $g$ and almost symmetric numerical semigroups with Frobenius number $F$  and type $F-2g$, for any $F\ge 4g+1$. This, in particular, provides an easy way to construct examples of almost symmetric numerical semigroups. And also opens a new way to restate Bras' conjectures on the number of numerical semigroups with genus $g$ \cite{bras}; this number is usually denoted by $n_g$. 

An important peculiarity of almost symmetric numerical semigroups with Frobenius number $F$ and type $F-2g$, with $F\ge 4g+1$ (for some nonnegative integer $g$), is that that these semigroups are uniquely determined by its sets of pseudo-Frobenius numbers (Corollary \ref{Cor UniPF}). This is in general is far from being true \cite{pseudo}, even for almost symmetric numerical semigroups, and it allows us to develop a new and faster algorithm for computing almost symmetric numerical semigroups with Frobenius number $F$ and type $F-2g$, with $F\ge 4g+1$. Thus, our approach can be potentially used to go further in the calculation of unknown elements of the sequence $n_g$.

\section{The correspondence}

The definition of almost symmetric numerical semigroups can be stated as follows (see for instance \cite{nari}).

\begin{definition}
A numerical semigroup $S$ is {almost symmetric} if for any integer $a$ not in $S$, then $\operatorname{F}(S)-a \in S\setminus\{0\}$ or $a \in \mathrm{PF}(S)$. 
\end{definition}

We will use the fact that the above definition is equivalent to \begin{equation}\label{eq:gft}
\operatorname{g}(S)=\frac{\operatorname{F}(S)+\operatorname{t}(S)}2.
\end{equation}

Motivated by the notion of set of gaps of a numerical semigroup, in \cite{EF} the concept of gapset is introduced.

\begin{definition}
A \textbf{gapset} is a finite subset $G \subset \mathbb{N} \setminus \{0\} $ such that given $a$ and $b \in \mathbb{N}\setminus\{0\}$ with $a+b \in G$, then either $a \in G$ or $b \in G$.
\end{definition}

Notice that if $G$ is a gapset, then $\mathbb{N} \setminus G$ is a numerical semigroup. We are going to give families of gapsets that  ``produce'' almost symmetric numerical semigroups.

\begin{proposition}\label{Prop main}
Let $S$ be a numerical semigroup with genus $g$ and let $F$ be a positive integer greater than $2 \operatorname{F}(S)$. The set 
\[G = \{1, \ldots, F\} \setminus \{F - a\ \mid\ a \in \mathbb{N} \setminus S\}\] is the gapset of an almost symmetric numerical semigroup with Frobenius $F$, type $F - 2g$ and multiplcity $F-f$. Moreover, 
\begin{align*}
\mathrm{PF}(\mathbb N \setminus G)  = & \{a \in S \mid 0 < a \leq \operatorname{F}(S)\} \\ & \cup\ \{\operatorname{F}(S)+1, \ldots, F-\operatorname{F}(S)-1\} \\ & \cup \ \big\{F - a\ \mid\ a \in S \cap \{0, \ldots \operatorname{F}(S)\} \big\}.
\end{align*}

\end{proposition}

\begin{proof}
Set $f = \operatorname{F}(S)$. First of all, we observe that \[G= \{1, \ldots, F-f-1\} \cup \big\{F - a\ \mid\ a \in S \cap \{0, \ldots f\} \big\}.\] Let us see that $G$ is actually a gapset. To do that we consider  $a$ and $b \in \mathbb{N}\setminus\{0\}$ with $a+b \in G$; in particular, $a+b \leq F$. Now, if $a \leq F-f-1$ or $b \leq F-f-1$, then we are done because this would imply $a$ or $b$ in $G$. So, let us assume $a \geq F-f$ and $b \geq F-f$; in this case, $a+b \geq 2F - 2f = F + (F-2f)$ and, since $(F-2f) > 0$, we conclude that $a+b > F$, which is incompatible with the condition $a+b \in G$. 

Let $S'$ be the numerical semigroup $\mathbb{N} \setminus G$; notice that $\operatorname{F}(S') =F$. Given $a \in G$, if $F-a \in S'$, then $F=a+(F-a) \not\in S'$, that is to say, $a \not\in \operatorname{PF}(S')$. Thus, to see that $S'$ is almost symmetric, we need to prove that given $a \in G$ with $F-a \in S$, then $a \in \mathrm{PF}(S')$. So, let $a \in G$ such that $F-a \in G$. First, we claim that $a \in S;$ indeed, if $F-a \leq F-f-1$, then $a \geq f+1$ and therefore $a \in S$. Otherwise, $F-a \in \{F-f, \ldots, F \} \cap \{F - a\ \mid\ a \in S\}$ and, clearly, $a \in S$. Now, let $b \in S' \setminus \{0\}$ and let us prove that $a + b \in S'$. If $a +b > F$, there is nothing to prove. Otherwise, $a+b = F - c$, for some $c \in \mathbb{N}$. If $a+b \not\in S'$, then $F-c \in G$. Arguing as above, we can prove that $c \in S$. Thus, since $a+c \in S$, we conclude that $b = F - (a+c) \in G$, in contradiction with $b \in S'$.

So, we have that $S' = \mathbb{N} \setminus G$ is an almost symmetric numerical semigroup. Moreover, since $\operatorname{g}(S') = \# G = F - g$ and $S'$ is almost symmetric, we have that $\operatorname{t}(S') = 2 \operatorname{g}(S') - F =  F-2g.$ We also observe that the smallest positive integer not in $G$ is $F-f$, so the multiplicity of $S$ is $F-f$.

Finally, let us see that $\mathrm{PF}(S')  = \{a \in S \mid 0 < a \leq f\} \cup\ \{f+1, \ldots, F-f-1\} \cup \ \big\{F - a\ \mid\ a \in S \cap \{0, \ldots f\} \big\}$. First, we notice that the right hand side has cardinality $(f-g)+(F-2f-1)+(f+1-g) = F-2g$. Thus, it suffices to see that all the elements in the right hand side are in $\mathrm{PF}(S')$. 
Notice that for every $a$ in the right hand side, $a\not\in S'$, and also $F-a\not\in S'$, thus 
 $a \in \mathrm{PF}(S')$ by the definition of almost symmetric numerical semigroup.
\end{proof}

\begin{definition}
Let $S$ be a numerical semigroup. A \emph{relative ideal} $I$ of $S$ is subset of $\mathbb{Z}$ such that 
\begin{enumerate}
\item $I+S \subseteq I$; 
\item $a + I \subseteq S,$ for some $a \in S$.
\end{enumerate}
\end{definition}

If $S$ is a numerical semigroup and $s \in \mathbb{Z}$, then \[K_S(s) := \{\operatorname{F}(S)+s - z \mid z \in \mathbb{Z} \setminus S\}\] is a relative ideal of $S$. This ideal is the $s-$shifted canonical ideal of $S$.

\begin{lemma}\label{lemma1}
If $S$ is a numerical semigroup and $F$ is a positive integer greater than $\operatorname{F}(S)$, then $\mathbb{N} \setminus K_S(F-\operatorname{F}(S)) = \{1, \ldots, F\} \setminus \{F-a \mid a \in \mathbb{N} \setminus S\}$.
\end{lemma}

\begin{proof}
Since $K_S(F-\operatorname{F}(S)) = \{F - z \mid z \in \mathbb{Z} \setminus S\} = \{F - a \mid a \in \mathbb{N} \setminus S\} \cup \{F+1, \ldots \}$, then $\mathbb{N} \setminus K_S(F-\operatorname{F}(S)) = \{1, \ldots, F\} \setminus \{F - a\ \mid\ a \in \mathbb{N} \setminus S\}$.
\end{proof}

\begin{theorem}\label{Th main}
Let $F$ and $g$ be positive integers such that $F \geq 4g-1$. The correspondence
\[S \mapsto K_S(F-\operatorname{F}(S))\] is a bijection between the set of numerical semigroups with genus $g$ and the set of almost symmetric numerical semigroups with Frobenius number $F$ and type $F -2g$. 
\end{theorem}

\begin{proof}
Let $S$ be a numerical semigroup with genus $g$.
By Lemma \ref{lemma1}, $\mathbb{N} \setminus K_S(F-\operatorname{F}(S)) = \{1, \ldots, F\} \setminus \{F - a\ \mid\ a \in \mathbb{N} \setminus S\}$. Moreover, since $\operatorname{F}(S) \leq  2g-1$ (see, for instance \cite[Lemma 2.14]{ns-book}), then $F > 4g-1 > 2\operatorname{F}(S)$. So, by Proposition \ref{Prop main}, our correspondence is a well defined application between the set of numerical semigroups with genus $g$ and the set of almost symmetric numerical semigroups with Frobenius number $F$ and type $F -2g$. Moreover, since $K_{S_1}(F-\operatorname{F}(S_1)) = K_{S_2}(F-\operatorname{F}(S_2))$ if and only if $S_1 = S_2$, we have that our application is clearly injective.

In order to see that it is surjective, let us consider an almost symmetric numerical semigroup $S'$ with Frobenius number $F$ and type $F -2g$, and set $G' := \{x \in \mathbb{N} \setminus S'\ \mid F-x \in S'\}$. We prove that $G'$ is a gapset. Take $a$ and $b \in \mathbb{N} \setminus \{0\}$ such that $a + b \in G',$ that is, $a+b \not\in S'$ and $F-(a+b) \in S'$, and let us prove that $a \in G'$ or $b \in G'$. To this end, we distinguish two cases:
\begin{itemize}
\item If $a \in S'$ or $b \in S'$. If  $a \in S'$, then $F-b = (F-(a+b))+a \in S'$ and $b \not\in S';$ otherwise $F = (F-b)+b \in S'$. Thus, $b \in G'$. By arguing analogously, if $b \in S'$, we obtain $a \in G'.$
\item If $a$ and $b \not\in S'$; in particular, we have that either $F-b \in S'$ or $b \in \mathrm{PF}(S')$, because $S'$ is almost symmetric. In the first case, $b \in G'$ and, in the second case, $F-a = b+(F-(a+b)) \in S,$ implies that $a \in G'$.
\end{itemize}
Thus, we have that $T=\mathbb{N}\setminus G'$ is a numerical semigroup. 
Since $S'$ is almost symmetric, by \eqref{eq:gft}, $\operatorname{g}(S')=(F+F-2g)/2=F-g$, and from the definition of almost symmetric numerical semigroup, we have that $\operatorname{g}(T)=\#G' = \operatorname{g}(S') - \#\mathrm{PF}(S') = (F-g)-(F-2g) = g$. Also  \[ K_T(F-\operatorname{F}(T)) = \{F-a\ \mid a \in G'\} \cup \{F+1, \ldots \}.\] 
From the definition of $G'$, $K_T(F-\operatorname{F}(T))\subseteq S'$. For the other inclussion, take $s\in S'$, with $s<F$. Then $F-s\in G'$, and $s=F-(F-s)\in K_T(F-\operatorname{F}(T))$.
\end{proof}

The inverse map in Theorem \ref{Th main} can be also explicitly described as we see next. If $S$ is a numerical semigroup, we will write $S^* = S \cup \mathrm{PF}(S)$. It easy to see that $S^*$ is a relative ideal of $S$. The ideal $S^*$ is called the \emph{dual of $S\setminus\{0\}$ with respect to $S$}. This term is justified by the fact that $S^*$ is equal to $\{z \in \mathbb{Z}\mid z + S \setminus \{0\} \subseteq S\}$.

Let us see that this star operation is the inverse of our application in Theorem \ref{Th main}.

\begin{proposition}\label{reverse}
Let $S$ be a numerical semigroup with genus $g$. If $F$ is a positive integer greater than $2 \operatorname{F}(S)$, then 
\[K_S(F-\operatorname{F}(S))^* = S.\]
\end{proposition}

\begin{proof}
By definition and the condition $F>\operatorname{F}(S)$, $K_S(F-\operatorname{F}(S)) \subseteq S$. By Proposition \ref{Prop main} and the proof of Theorem \ref{Th main}, and using once more that $F>\operatorname{F}(S)$, we have that 
$\mathrm{PF}(K_S(F-\operatorname{F}(S))) \subseteq S$. Thus, $K_S(F-\operatorname{F}(S))^* \subseteq S$. Now, $\#(\mathbb{N}\setminus K_S(F-\operatorname{F}(S))^*)=\operatorname{g}(K_S(F-\operatorname{F}(S)))-\operatorname{t}(K_S(F-\operatorname{F}(S)))$, and according to the proof of Theorem \ref{Th main}, this amount equals  $F-\operatorname{g}(S)-(F-2\operatorname{g}(S))=\operatorname{g}(S)$. We conclude that $K_S(F-\operatorname{F}(S))^* = S$. 
\end{proof}

This correspondence provides a new characterization of almost symmetric numerical semigroups with high type.

\begin{corollary}\label{cor:char-almos-ht}
 Let $T$ be a numerical semigroup with Frobenius number $F$ and type $t$, with $t\ge (F-1)/2$ and $F-t$ even. Then $T$ is almost symmetric if and only if $T^*$ is a numerical semigroup with genus $(F-t)/2$.
\end{corollary}

\begin{proof}
\emph{Necessity}. If $T$ is almost symmetric, as $F-t$ is even, then $F-t=2g$ for some nonnegative integer $g$, and $t\ge (F-1)/2$ yields $F\ge 4g-1$. Thus, $T= K_S(F-\operatorname{F}(S))$ for some numerical semigroup $S$ of genus $g$ (Theorem \ref{Th main}). Notice that $2g-1\ge \operatorname{F}(S)$, whence $2\operatorname{F}(S)\le 4g-2< F$. By Proposition \ref{reverse}, we conclude that $T^*=S$, and we are done.

\noindent\emph{Sufficiency}. Let $S=T^*$, and set $g=(F-t)/2$, the genus of $S$.
As $S=T^*=T\cup\operatorname{PF}(T)$, we have $g=\operatorname{g}(T)-t$, and thus 
\[\operatorname{g}(T)=\frac{F-t}2+t=\frac{\operatorname{F}(T)+\operatorname{t}(T)}2,\] proving that $T$ is almost symmetric.
\end{proof} 

The depth of a numerical semigroup has shown to play an special role in the study of Wilf's conjecture (see for instance \cite{manuel-survey, EF}). Let $S$ be a numerical semigroup with Frobenius number $F$ and multiplicity $m$, and write $F+1=qm-r$ for some integers $q$ and $r$ with $0\le r <m$. Then its \emph{depth} is $\operatorname{depth}(S)=q$.

\begin{corollary}
Let $S$ be a numerical semigroup with genus $g$ and let $F$ be a positive integer greater than $2 \operatorname{F}(S)$. The semigroup $K_S(F-\operatorname{F}(S))$ has depth equal to two.
\end{corollary}
\begin{proof}
By Lemma \ref{lemma1}, the multiplicity of $K_S(F-\operatorname{F}(S))$ is equal to $F-\operatorname{F}(S)$. Write $F+1=2(F-\operatorname{F}(S))-(F-\operatorname{F}(S)-1)$. Then $\operatorname{depth}(S)=2$.
\end{proof}

Depth equal to two has a particular relevance, since Bras' conjecture holds in the restricted class of numerical semigroups having this depth \cite{EF}.

%%%%%%%%%%%%%%%%%%%%%%%%%%%%%%%%%%%%%%%%%%%%%%%%%%%%%%%%%%%
%%%%%%%%%%%%%%%%%%%%%%%%%%%%%%%%%%%%%%%%%%%%%%%%%%%%%%%%%%%
\section{The algorithm}

Write $\mathscr{A}(F,t)$ for the set of almost symmetric numerical semigroups with Frobenius number $F$ and type $t$, and let $n_g$ be the number of numerical semigroups with genus $g$. As an immediate consequence of Theorem \ref{Th main} we obtain the following result.

\begin{corollary}\label{Cor IMNS18}
Let $g\in \mathbb{N}$, and let $F$ be a an integer greater than or equal to $4g-1$. Then number of numerical semigroups with genus $g$ is equal to the 
the number of almost symmetric numerical semigroups with Frobenius number $F$ and type $F-2g$. That is, \[F \geq 4g-1  \hbox{ implies } \# \mathscr{A}(F,F-2g) = n_g.\]
\end{corollary}

With this corollary we can restate the weaker version of the conjecture appearing in \cite{bras},  that is, that the sequence $n_g$ is increasing. Notice that in order to prove that $n_{g+1}>n_g$, one needs to show that 
\[\# \mathscr{A}(F,F-2(g+1))> \# \mathscr{A}(F,F-2g),\]
for $F$ large enough. This opens a new perspective to attack this conjecture.

We recall that the largest $n_g$ known so far appears in 

\centerline{\url{https://github.com/hivert/NumericMonoid}.}

\medskip
As we mentioned in the introduction, almost symmetric numerical semigroups with high type are uniquely determined by their sets of pseudo-Frobenius numbers.

\begin{corollary}\label{Cor UniPF}
Let $F \in \mathbb{N}$ and let $t \geq (F-1)/2$. Every almost symmetric numerical semigroup with Frobenius number $F$ and type $t$ is uniquely determined by its pseudo-Frobenius numbers.
\end{corollary}

\begin{proof}
Let $S'_1$ and $S'_2$ be two almost symmetric numerical semigroups with Frobenius number $F$ and type $t$ such that $\mathrm{PF}(S'_1) = \mathrm{PF}(S'_2)$. By \eqref{eq:gft}, the genus of $S_1'$ and $S_2'$, equals $(F+t)/2$; whence $F-t=F+t-2t$ is even. Set $2g=F-t$. Then $2F\ge 4g+2t\ge 4g+F-1$, which implies $F \geq 4g-1$. So, by Theorem \ref{Th main}, there exist unique numerical semigroups $S_1$ and $S_2$ of genus $g$ such that $S'_i = K_{S_i}(F-\operatorname{F}(S_i))$, $i \in\{ 1,2\}$. Observe that $F \geq 4g-1 \geq \operatorname{F}(S_1)+\operatorname{F}(S_2)+1$, because as we already mentioned above $F(S_i) \leq 2g-1$, $i\in\{ 1,2\}$ (see, for instance \cite[Lemma 2.14]{ns-book}). Moreover, without loss of generality we may suppose that $\mathrm{F}(S_1) \leq \mathrm{F}(S_2)$.  Thus, by the last part of Proposition \ref{Prop main}, we have that $\{a \in S_1 \mid 0 < a \leq \operatorname{F}(S_1)\} \cup \{\operatorname{F}(S_1) + 1, \ldots, \operatorname{F}(S_2)\} \subseteq \{a \in S_2 \mid 0 < a \leq \operatorname{F}(S_2)\}$, which implies $S_1 \subseteq S_2$ or equivalently, $\mathbb{N} \setminus S_2 \subseteq \mathbb{N} \setminus S_1$. Now, since both $\mathbb{N} \setminus S_2$ and $\mathbb{N} \setminus S_1$ have cardinality $g$, we conclude that $S_1 = S_2$, and consequently $S'_1 = S'_2$.
\end{proof}

Notice that in general a potential set of pseudo-Frobenius numbers does not uniquely determine a numerical semigroup, \cite{pseudo}, even under the almost symmetric condition. It may happen that several numerical semigroups share the same set of pseudo-Frobenius numbers. Thus the above result opens a new strategy to determine almost symmetric numerical semigroups with large type with respect to the Frobenius number.

\begin{example}
There are $103$ almost symmetric numerical semigroups with Frobenius number $20$, while when computing their sets of psuedo-Frobenius numbers, we only get $62$ different possible sets.
{\small
\begin{verbatim}
 gap> l:=AlmostSymmetricNumericalSemigroupsWithFrobeniusNumber(20);;
 gap> Length(l);
 103
 gap> Length(Set(l,PseudoFrobenius));
 62    
\end{verbatim}
}
\end{example}

In \cite[Theorem 4.1]{as-ft} it is show that each almost symmetric numerical semigroup of Frobenius $F$ and type $t$ is obtained from an almost symmetric numerical semigroup of Frobenius $F$ and type $t+2$. More precisely, the following is proved.

\begin{theorem}\label{16}
Let $F \geq 5$ and $t+2 \leq F$ be a positive integer greater than $2$ such that $F+t$ is even. Then $S'$ is an almost symmetric numerical semigroup with Frobenius number $F$ and type $t$ if and only if there exist an almost symmetric numerical semigroup $S$ with Frobenius number $F$ and type $t+2$, and $i \in \{t+1, \ldots, \operatorname{m}(S)-1\}$ such that 
\begin{enumerate}[(a)]
\item $S' = \{i\} \cup S$,
\item \label{item:b} $-i + \mathbb{N} \setminus S' \subseteq \mathbb{Z} \setminus S'$, and
\item \label{item:c} $\big(i + (\mathrm{PF}(S) \setminus \{i,F-i\}) \big) \subseteq S'$.
\end{enumerate}
In this case, $\mathrm{PF}(S) = \mathrm{PF}(S') \cup \{i,F-i\}.$
\end{theorem}

Let us see that if $t \geq (F-1)/2$,  conditions (\ref{item:b}) and (\ref{item:c}) in Theorem \ref{16} can be replaced by a simpler test.

\begin{proposition}\label{Prop algo}
Let $F \geq 5$ and let $t \in [(F-1)/2,F-2]$ be an integer such that $F+t$ is even. Then $S'$ is an almost symmetric numerical semigroup with Frobenius number $F$ and type $t$ if and only if there exist an almost symmetric numerical semigroup $S$ with Frobenius number $F$ and type $t+2$, and $i \in \{t+1, \ldots, \operatorname{m}(S)-1\}$ such that 
\begin{enumerate}[(a)]
\item $S' = \{i\} \cup S$, and
\item $(i+(\operatorname{PF}(S) \setminus \{i,F-i\})) \cap (\operatorname{PF}(S) \setminus \{i,F-i\}) = \varnothing$
\end{enumerate} 
In this case, $\mathrm{PF}(S') = \mathrm{PF}(S) \setminus \{i,F-i\}$ and $S'$ is the unique almost symmetric numerical semigroup with set of pseudo-Frobenius numbers equal to $\mathrm{PF}(S) \setminus \{i,F-i\}$.
\end{proposition}

\begin{proof}
Necessity follows from Theorem \ref{16}. 

For the other implication, first observe that as $S$ is almost symmetric, $F-(t+2)$ is even. So, $g = (F-(t+2))/2$ is a nonnegative integer. Moreover, $F \geq 4g-1$, because $t+2 \geq t \geq (F-1)/2$. Thus, by Theorem \ref{Th main}, there exists a numerical semigroup $T$ of genus $g$ such that $K_T(F-\operatorname{F}(T)) = S$. Now, on the one hand, by Proposition \ref{Prop main}, we have that $\operatorname{m}(S) = F-\operatorname{F}(T)$ and, on the other hand, we have that, by \cite[Lemma 2.14]{ns-book}, $\operatorname{F}(T) \leq 2g - 1 \leq (F-1)/2 \leq t$. Therefore, $\{t+1, \ldots, \operatorname{m}(S)-1\} \subseteq \{\operatorname{F}(T)+1, \ldots, F-\operatorname{F}(T)-1\}$, in particular, $\{t+1, \ldots, \operatorname{m}(S)-1\} \subseteq \operatorname{PF}(S)$ in light of Proposition \ref{Prop main}. So, both $i$ and $F-i$ are pseudo-Frobenius numbers of $S$ and $\mathrm{PF}(S) \setminus \{i,F-i\}$ has cardinality $t$.

Consider now $S' = S \cup \{i\}$ and set $PF':=\mathrm{PF}(S) \setminus \{i,F-i\}$. Since $i \in \operatorname{PF}(S)$ and $2i \in S$ because $2i \geq 2t+2 \geq F$. We have that $S'$ is a numerical semigroup and $\operatorname{F}(S') = F$, because $i < F$. Moreover, $PF' + a \in S \subset S'$, for every $a \in S \setminus \{0\}$, because $PF' \subset \operatorname{PF}(S)$. Therefore, if $PF' + i \subset S'$, then we have $PF'\subseteq \operatorname{PF}(S')$.

Suppose $i+PF' \not\subset S'$. Then there exists $a \in PF' \subset \operatorname{PF}(S)$ such that $i+a \not\in S'$. Thus, $a+i \in \operatorname{PF}(S)$. Finally, as $i+a \neq i$, because $a \neq 0$, and $i+a \neq F-i$, because $2i > F$, we conclude that $i+a \in PF'$; in contradiction with condition (b).

Now we have $PF'\subseteq \operatorname{PF}(S')$, and we know that $\operatorname{g}(S')\ge (\operatorname{F}(S')+\operatorname{t}(S'))/2$ \cite[Proposition 2.2]{nari}. Consequently, $t\le \operatorname{t}(S')\le 2(\operatorname{g}(S)-1)-\operatorname{F}(S)=\operatorname{t}(S)-2=t$. Whence $t=\operatorname{t}(S')$ and by \eqref{eq:gft}, we deduce that $S'$ is almost symmetric. 

The uniqueness of $S'$ follows from Corollary \ref{Cor UniPF}.
\end{proof}

Let us see that the above results offer an alternative to compute $n_g$.

\begin{algorithm}
The following GAP \cite{gap} code counts the number of almost symmetric numerical semigroups of Frobenius number $F=4g-1$ and type $F-2j$ for each $j \in\{ 1, 2, \ldots, \lceil (F-1)/4 \rceil\}$; equivalently, by Corollary \ref{Cor IMNS18}, the number of numerical semigroups with genus $j,$ for each $j\in\{1, \ldots, g\}$. 
The main function is nothing but a recursive step that calls to the \texttt{auxiliar} function whose correctness relies in Proposition \ref{Prop algo}. We observe that since the pseudo-Frobenius numbers uniquely determine a numerical semigroup by Corolary \ref{Cor UniPF}, we only need to deal with pseudo-Frobenius sets. Moreover,  by Corolary \ref{Cor UniPF} again, in order to avoid unnecessary repetitions we can restrict the upper range of $i$ in Proposition \ref{Prop algo}, the mentioned restriction is forced with the second argument of the \texttt{auxiliar} function.

It is important to emphasize that our GAP code does not require to make calls to other libraries or GAP packages. This makes our method more versatile and suitable to be implemented in other programming languages.

{\small
\begin{verbatim}
 auxiliar := function(PF,m,t,s)
  local L,F,PF1,i,k,j;
   L := []; F:=PF[t+2];
   for i in [t+1 .. m-1] do
    PF1:=Difference(PF,[i,F-i]);
    k:=0;
    for j in [1 .. s] do
     if ((PF1[j]+i) in PF1) then
      k:=1;
      break;
     fi;
    od;
    if k=0 then 
     Append(L,[[PF1,i]]);
    fi;
   od;
   return L;
 end;
 
 counting_function := function(g)
  local F,L,j,M,t,s,N;
  F:=4*g-1;
  L:=[[[1 .. F] ,F]];
  for j in [1 .. g] do
   M:=[];
   t:=Length(L[1][1])-2; 
   s:=Int(t/2);
   for N in L do 
    Append(M,auxiliar(N[1],N[2],t,s));
   od;
   L:=M;
   Print("n", j, " = ", Length(L), "\n");
   Unbind(M);GASMAN("collect"); #Cleaning Memory
  od;
  return Length(L);
 end;
\end{verbatim}
}
\end{algorithm}

A quick comparison with the following GAP command (included in the GAP package  \texttt{numericalsgps} \cite{numericalsgps}) 

\centerline{\texttt{Length(NumericalSemigroupsWithGenus(g))}} 
\noindent evidences that our code is slightly faster for $g \geq 23$. For instance, if $g = 26$ our function, \texttt{counting\_function(26)}, computes $[n_1, \ldots n_{26}]$ in $17.335$ seconds, while the above command takes $19.972$ seconds to compute $n_{26}$. Both computations have been performed running GAP in a Intel(R) Core(TM) i7-4770S CPU 3.10 GHz. This simple evidence opens a door to more efficient and faster implementations.

Finally, we observe that, by Proposition \ref{Prop main} and Corollary \ref{Cor UniPF}, we can take advantage of the function \texttt{NumericalSemigroupByGaps} included in the GAP package  \texttt{numericalsgps} to recover the whole set semigroups of genus $g$ from our code. This can be done by just replacing \texttt{return Length(L);} with
\begin{verbatim}
 return List(L, j->NumericalSemigroupByGaps(
                  Difference([1 .. (2*g-1)],j[1])));
\end{verbatim}  
With this modification the computation of the whole set of numerical semigroups of genus $g$ took $30.351$ seconds.

\medskip
\noindent\textbf{Acknowledgements.}
The authors would like to thank F\'elix Delgado for his constructive comments. This note was written during a visit of the second author to the IEMath-GR (Universidad de Granada, Spain), he would like to thank this institution for its hospitality. Corollary \ref{Cor IMNS18} was conjectured by the second author at the INdAM meeting: ``International meeting on numerical semigroups - Cortona 2018'', he would like to thank the organizers for such a nice meeting.

%%%%%%%%%%%%%%%%%%%%%%%%%%%%%%%%%%%%%%%%%%%%%%%%%%%%%%%%%%%
%%%%%%%%%%%%%%%%%%%%%%%%%%%%%%%%%%%%%%%%%%%%%%%%%%%%%%%%%%%

\end{document}